\documentclass[reqno,11pt]{amsart}
 \usepackage[square,sort,comma,numbers]{natbib}
\usepackage{palatino}
\usepackage{amsfonts,amsmath} 
\usepackage{amsthm,soul}
\usepackage{graphicx,enumerate}
\usepackage[hmarginratio=1:1,hscale=0.75]{geometry}
\usepackage[latin1]{inputenc}
\usepackage[all]{xy}

\usepackage{amssymb,mathtools}
\usepackage{stmaryrd}
\usepackage{hyperref}
\usepackage{pdfsync}
\usepackage{color}
\usepackage[T1]{fontenc} 
\usepackage{lmodern}

\newtheorem{counter}{Counter}[section]
\newtheorem{lem}[counter]{Lemma}

\newtheorem{thm}[counter]{Theorem}

\newtheorem{prop}[counter]{Proposition}

\newtheorem{conj}[counter]{Conjecture}

\newcommand{\R}{\mathbb{R}}

\newcommand{\lra}{\longrightarrow} 
\newcommand{\Lra}{\Longrightarrow}
\newcommand{\ra}{\rightarrow}

 \newcommand{\sse}{\subseteq}
\renewcommand{\~}{\tilde}

 \newcommand{\sh}{\sharp}
\newcommand{\pd}{\partial}

\newcommand{\fal}{\forall}
\newcommand{\8}{\infty}

\newcommand{\vep}{\varepsilon} 

\newcommand{\al}{\alpha}
 \newcommand{\bt}{\beta}
\newcommand{\gm}{\gamma}
 \newcommand{\ggm}{\Gamma}

 \newcommand{\om}{\Omega}
 
\renewcommand{\vartheta}{\Theta}

\DeclareMathOperator{\dist}{dist}

\DeclareMathOperator{\diam}{{\operatorname{diam}}}
\DeclareMathOperator{\argmin}{{\operatorname{argmin}}}

{   \end{list} }

\definecolor{mygreen}{rgb}{0.1,0.75,0.2}

\renewcommand{\d}{\,{\operatorname{d}}}

\begin{document}
\date{}

\title[Optimal Centroidal Voronoi Tesselations in 3D]{Bounds on the Geometric Complexity of Optimal Centroidal Voronoi Tesselations in 3D}

\author{Rustum Choksi}
\address{Department of Mathematics and Statistics, McGill University, Montr\'{e}al, QC Canada}
\email{rustum.choksi@mcgill.ca}
\author{Xin Yang Lu}
\address{Department of Mathematical Sciences, Lakehead University, Thunder Bay, ON, Canada and Department of Mathematics and Statistics, McGill University, Montr\'{e}al, QC Canada}
\email{xlu8@lakeheadu.ca}

\begin{abstract}
Gersho's conjecture in 3D asserts the asymptotic periodicity and structure of the optimal centroidal Voronoi tessellation. This relatively simple crystallization problem remains to date open. We prove bounds on the geometric complexity of optimal centroidal Voronoi tessellations as the number of generators tends to infinity. Combined with an approach introduced by Gruber in 2D, these bounds reduce 
the resolution of the 3D Gersho's conjecture to a finite, albeit very  large,  computation
 of an explicit convex problem in finitely many variables.

 \end{abstract}

\maketitle


\textbf{Classification. }
52C35, 52C45, 52C07, 49Q20, 82D25. 

\section{Introduction}

A fundamental problem (cf. \cite{DFG, CS, Gr3}) in  both information theory and 
discrete geometry  is known, respectively,  as {\it optimal block quantization} or {\it optimal centroidal Voronoi tessellations} (CVT).  To state the problem, consider a bounded domain in $\R^N$, say  a cube  $Q=[0,1]^N$, and for a collection of points $y_k \in Y=\{y_1,\cdots,y_n\}\sse Q$,  define the associated Voronoi regions (comprising a Voronoi tessellation of $Q$)
\[ V_k  \, = \, \{ x \in Q \, | \, |x - y_k| \le |x - y_i| \,\, \forall \,\, i \ne k \}.\]
A 2D illustration with $n = 6$ is presented on the left of Figure \ref{fig2}. 
\begin{figure}
  {\includegraphics[width=0.3\textwidth]{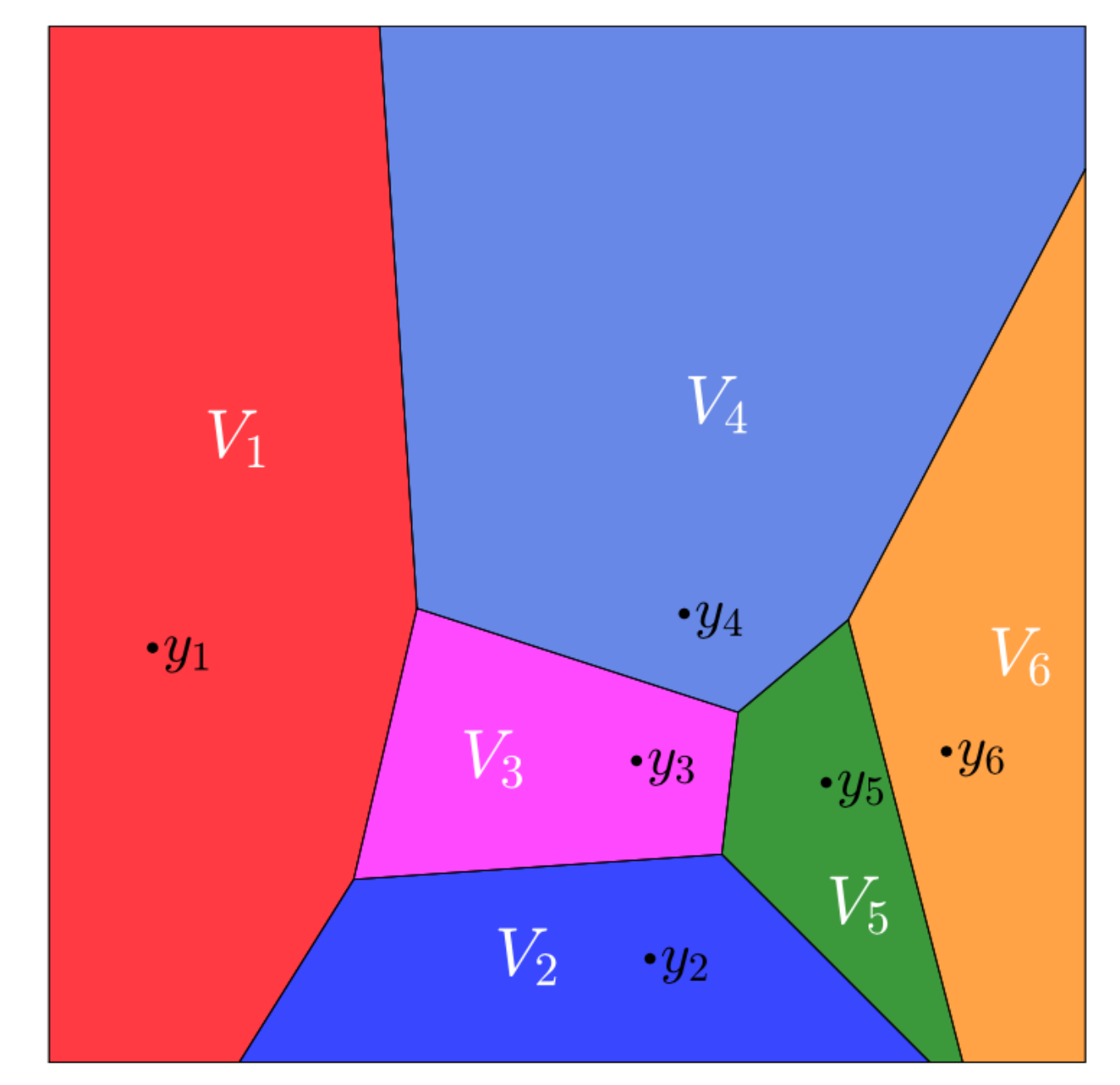}}  \hspace{-0.38cm}
    {\includegraphics[width=0.26\textwidth]{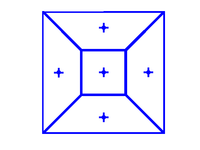}}\hspace{-0.88cm} {\includegraphics[width=0.26\textwidth]{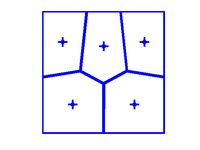}} \hspace{-0.88cm}{\includegraphics[width=0.26\textwidth]{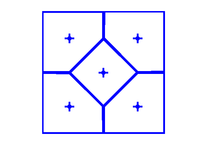}}
        \caption{Left: A Voronoi diagram (the Voronoi regions associated with six generators). Right: Three centroidal Voronoi tessellations with five generators.} 
\label{fig2}
\end{figure}
A centroidal Voronoi tessellation (cf. Figure \ref{fig2} right) 
amounts to finding a placement of the points $y_k$ such that 
they are exactly the centroids of their associated Voronoi region $V_k$.  A variational formulation is based upon minimization of the following nonlocal energy  
\begin{equation}\label{Vor-energy}
E(Y):=\int_Q\dist^2(x,Y)\d x\, =\, \sum_{k=1}^n \int_{V_k} |x-y_k|^2\d x. 
\end{equation}
Criticality of $E$ is exactly the condition that each $y_k$ be the centroid of  its Voronoi region $V_k$, that is 
\[ Y^\ast = \{ y_i^\ast\} \,\,\,\, \hbox{\rm is a critical point of $E$}  \qquad {\rm iff} \qquad 
y_i^\ast \, = \, \int_{V_i} x \, dx, \quad \hbox{\rm the centroid  of $V_i$}. \]
In the context of information theory, the set $Y$ is viewed as a {\it quantizer} to quantize data which is distributed in $Q$ according a  continuous probability density, here taken to be uniformly  distributed across $Q$. The {\it quantization error} is given by $E(Y)$.  The {\it optimal quantizer} is the one with least error, alternatively {\it the CVT with lowest energy} (\ref{Vor-energy}). 

A well-known conjecture attributed to Gersho \cite{Ge} (cf. Conjecture (\ref{Gersho}) (a) below) addresses the periodic nature of the configuration with least error (alternatively, the CVT with lowest energy). 
This conjecture is completely solved in 2D but, to date,  remains open in 3D.
We present a precise statement of Gersho's conjecture (statement (a)) in its augmented form (statement (b)): 

\begin{conj}\label{Gersho}  {\bf The Augmented Gersho's Conjecture} 
\begin{itemize}
\item[(a)] There exists a polytope $V$ with $|V|=1$ which tiles the space with congruent copies
such that the following holds: let $(Y_n)_n$ be a sequence of minimizers, with $Y_n\in\argmin_{\sh Y=n} E(Y)$, then the Voronoi cells  
of points $ Y_n$ are asymptotically congruent to  $n^{-1/N}V$ as $n\to +\8$.\\

\item[(b)]
For dimension $N=2$, the optimal polytope $V$ is a regular hexagon, corresponding to a 
 optimal placement of points on a triangular lattice (cf. Figure $\ref{fig1}$ left). For dimension $N = 3$, the optimal polytope $V$ is the truncated octahedron,  corresponding  an optimal placement of points on a BCC (body centered cubic) lattice   (cf. Figure $\ref{fig1}$ right). 
\end{itemize}
\end{conj}

\begin{figure}
  {\includegraphics[width=0.3\textwidth]{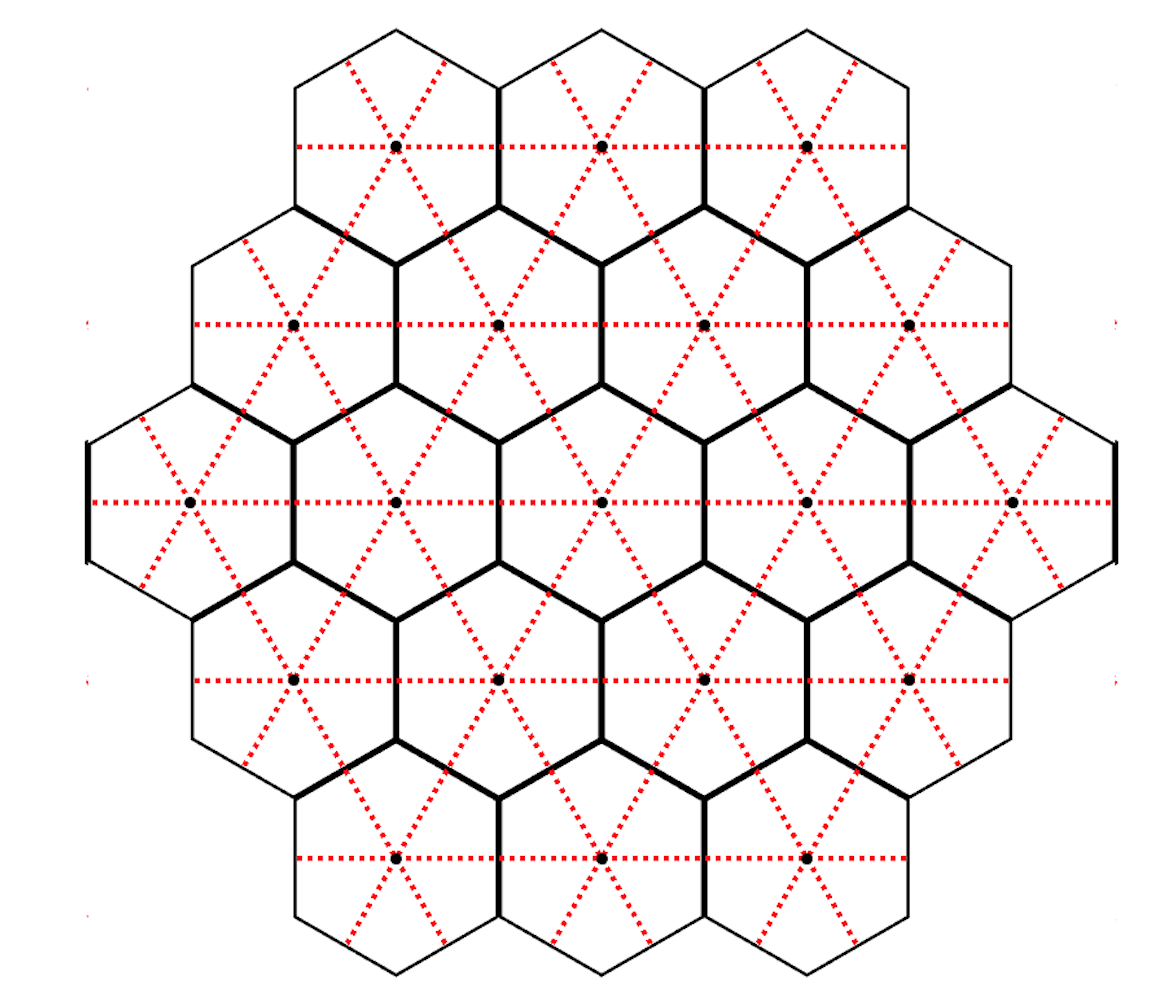}}\qquad \qquad 
    {\includegraphics[width=0.2\textwidth]{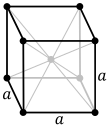}}\,\,\, \,\,\, 
      {\includegraphics[width=0.2\textwidth]{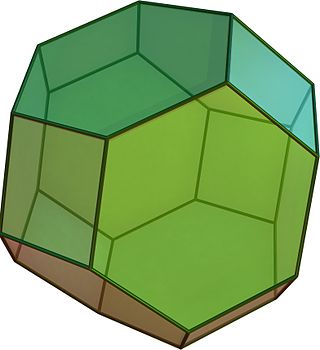}}
      \caption{Left: 2D Optimal placement of points on a triangular  lattice with associated  optimal Voronoi polytope  a regular hexagon. Right: 3D Conjectured optimal  placement of points on a BCC lattice and the associated   optimal Voronoi polytope the
      truncated octahedron (Source: Wikipedia).} 
\label{fig1}
\end{figure}

Conjecture \ref{Gersho} has been proven in 2D where 
hexagonal structures are pervasive\footnote{For example,  Hales' celebrated resolution of the Honeycomb Conjecture in \cite{Ha}.} .  The essential parts of the proof were first presented by Fejes T\'oth \cite{Toth} with later versions given, for example, by Newman \cite{Ne}. However, as noted in \cite{Gr1}, the first complete 2D proof of Gersho's conjecture was given by Gruber. To date, the conjecture remains open in 3D. In 3D,  Barnes and Sloan \cite{BS} have proven the optimality of the BCC configuration amongst all lattice configurations, while Du and Wang \cite{DW} have presented numerical evidence supporting the conjecture. 
The nonlocal and nonconvex character of  (\ref{Vor-energy})  insures a highly nontrivial energy landscape associated with a multitude of critical points with complex, albeit polygonal, Voronoi regions. Moreover, to divorce from boundary/size effects,  one can only address the asymptotics as the number of generators n tends to infinity. 

The purpose of this paper is to present in 3D some quantitative bounds for the geometry of minimizing Voronoi regions (cf. Theorem \ref{main1}). 
To our knowledge, these bounds are new. 
In particular, we prove an upper bound (independent of $n$) on the complexity (number of faces) of an optimal Voronoi cell. 
 This is an important step: 
Indeed, to divorce from boundary/size effects, one can only address the asymptotics as the number of generators $n$ tends to infinity. A priori, we cannot dismiss the possibility that the complexity of the Voronoi cells associated with a CVT  is $O(n)$ as $ n \to \infty$;  
  what we can do is to prove that this is not the case for the {\it optimal} CVT. 
As we explain in Section \ref{sec-Gersho}, we can combine this bound with Gruber's two dimensional approach
to reduce the 3D Gersho conjecture to a finite, albeit large, computation
 of an explicit convex problem in finitely many variables.
 
Remarkably,  the proof of these bounds 
does not rely on any sophisticated mathematical machinery, rather solely on 
elementary estimates with distance functions. Our choice of domain (the unit square $Q$) is for convenience only: the analogous results hold for any finite domain or, for example, the flat torus.

Let us conclude the introduction by noting that Gersho's conjecture is related to a fundamental, largely open, question in 
condensed matter physics. 
The {\it Crystallization Conjecture}  roughly states that within the confinements of some physical domain, $n$ interacting particles arrange themselves into a periodic configuration. Precisely, 
 let $\Omega$ be a domain in $\R^N$, $ Y=\{y_1,\cdots,y_n\}$
a collection of $n$ points in $\Omega$,  and define 
\begin{equation}\label{Pot-energy}
 F_{\mathcal V} (Y) \, : =\,  \sum_{i,j = 1}^n {\mathcal V} (|y_i - y_j|), \end{equation}
where ${\mathcal V}$ is the interaction potential. In this respect, the crystallization
 conjecture asserts that as $n \to \infty$, the minimizers  of $F_n$ over all possible points  
$y_1, \dots, y_n \in \Omega$ arrange themselves in a periodic lattice. 
Typical physical interaction potentials, for example the Lennard-Jones potential,  have the property that they are repulsive at short distances and attractive at large.  
To dispense with boundary effects, it is necessary to pose the problem as an asymptotic statement as the size of the domain get larger. Upon rescaling, this is equivalent to letting the number of particles 
 $n \to \infty$.  The crystallization conjecture remains  one of the most fundamental and difficult problems in mathematical physics with rigorous results far and few (see, for example, \cite{BL, Ra, HR, Theil, EL, FT, FES}).  
As noted in \cite {HR}, there is a direct link between the crystallization conjecture in 3D and sphere packing problems (cf.  \cite{Via8, Via24} for recent new developments.) 

The relationship of our purely geometric variational problem  (\ref{Vor-energy}) 
to the ubiquitous class (\ref{Pot-energy}) is not immediate; in the former, the points do interact with each other but implicitly, via the distance function (equivalently via the associated Voronoi regions). 
While there is no explicit effective interaction potential ${\mathcal V}$, one can reformulate the energy $E$ in terms of the {\it Wasserstein-2 distance} $W_2$ (cf. \cite{Vi}) between a weighed sum of delta functions and  Lebesgue measure $\mathcal{L}_N$: 
\[ E(Y) \, = \, W_2 \left(\sum_{i = 1}^n |V_i| \delta_{y_i} \, , \, \mathcal{L}_N \right)^2, \]
where  $\delta_{y_i}$ denotes the delta function with concentration at $y_i$. 
In other words, the quantization error is precisely the  squared Wasserstein-2 distance distance between the weighted point quantities and the continuous probability density. Such {\it semi-discrete} optimal transportation problems have recently been studied in \cite{BJW}. 

 In our opinion, the optimal CVT problem is the simplest setting to prove 3D crystallization because: 
\medskip
\begin{itemize}
\item there is a simple and elegant characterization of criticality (critical points); 
\medskip
\item working solely with distance functions facilitates 
the  proof of  estimates and quantitative bounds for optimal configurations entirely in terms of their convex polygonal Voronoi regions.  In particular,  the energy  (\ref{Vor-energy})  has a pseudo-local 
character which means that one can readily estimate the total energy  loss  resulting from the addition of a new generator in a fixed Voronoi cell (cf. Lemma \ref{d}).
\end{itemize}
\medskip

\section{Optimal CVT and Gersho's conjecture: Previous results and the statement of our main theorem}

\subsection{Gruber's Approach in Two Dimensions}

In \cite{Gr1}, Gruber presented an elementary  proof in 2D of Gersho's conjecture. For convenience, he took the  domain $\Omega$ to be a suitably-chosen regular $n$-gon; however, one can work on an arbitrary domain at the expense of smaller-order boundary errors. His argument is as follows:  
\begin{enumerate}
 \item[(i)] First, it is shown that the functional
 \begin{equation*}
  G(a,m):=\min_{\substack{A \text{ is an } \\ m\text{-gon with area } a}} \,\,\, \int_A |x-y|^2 \d x,\qquad y=\text{centroid of } A
 \end{equation*}
is convex in both variables.  Here, one first shows that the minimum is attained on regular polygons. Then, via a direct Hessian computation, it is shown that there exists an 
 an extension of  $G$, say   $\tilde{G}$, whose second argument is defined over the positive real numbers, which is convex in both variables.

 \item[(ii)] Second, it is shown that given a Voronoi tessellation $\{V_i\}_{i=1}^n$, the average number of sides
 is at most $6$: let $E(F)$ be the number of sides of the face $F$, and by double counting (each side belongs to exactly 2 faces) we get
$ \sum_{\{F \text{ faces}\}} E(F)=2e\le 6n-12, $
 where $e$ is the total number of sides, and $2e\le 6n-12$ comes from Euler's formula for polytopes. 
 Moreover, it is easy to check that 
 \[G(a,6)\le \min\{G(a,3),G(a,4),G(a,5)\}\]
 for all $a\ge 0$, by directly computing the values of $G(a,6), G(a,3),G(a,4),G(a,5)$ on regular $3,4,5,6$-gons.
  
 \item[(iii)] With these steps in hand, one proceeds as follows.   Let $\{V_i\}_{i=1}^n$ be an arbitrary Voronoi tessellation and  denote: by $s_i$ the number of sides
 of $V_i$,  by $a_i$ its area, and 
 \[ \-a:=\frac1n\sum_{i=1}^n a_i, \qquad \,\,\, 
  \ \-s:=\frac1n\sum_{i=1}^n s_i. \]
The convexity of $G$ then implies that 
 \begin{align*}
  \sum_{i=1}^n \int_{V_i} |x-y_i|^2\d x &\ge \, \sum_{i=1}^n G(a_i,s_i) \\
  &\ge\, nG(\-a,\-s) + o(n)\\
  & \ge\,  nG(\-a , 6) + o(n),
 \end{align*}
where $o(n)$ is the contribution of the boundary terms, which vanish as $n\to +\8$. The last inequality shows that
the hexagonal partition is optimal.
\end{enumerate}

\bigskip

 The fundamental difficulty of applying Gruber's arguments in 3D case is establishing the convexity in $m$ of 
\begin{equation*}
  G(a,m)\,:=\, \min_{\substack{V \text{ convex polytope, } |V|=a\\
V \text{ has at most } m \text{ faces }  }} \,\,\, \,  \int_V |x-y|^2 \d x,\qquad y=\text{ centroid of } V.
 \end{equation*}
We do not have regular $m$-hedron in 3D,
and computations are unfeasible. A priori, the maximum number of possible faces of 
the Voronoi polygons associated with a critical point can grow with $n$. One of the main results of this paper is to prove  (cf. Theorem \ref{main1}) upper bounds on the geometric complexity (including the number of 
faces) of such polygons which are independent of $n$. With such bounds in place, one could, in principle, have the computer verify the convexity of $G(a,m)$. As we explain in the last section (Section \ref{sec-Gersho}), this would then prove Gersho's conjecture in 3D. 

 Perhaps a deeper reason for the significantly increased difficulty in proving Gersho's in 3D, compared to
2D, is due to the fact that we do not expect the presence of a {\em universally optimal} configuration
(\cite[Definition~1.3]{Via uni}) in 3D. This is in stark contrast with the 2D case, where the triangular lattice
is almost surely to be universally optimal (cf. \cite{Via uni}), although no rigorous proof is available. Gersho's
conjecture would not be the first one in which such issue appears: it is well known that the solution
to the optimal foam problem in 2D is given the honeycomb structure, whose barycenters lie on the triangular lattice, while in 3D this is still open, and the long conjectured solution, i.e. the bitruncated cubic honeycomb, is surely not optimal, as it has higher energy than the Weaire-Phelan structure (cf. \cite{WP}).

Before presenting our results, let us document two known results in 3D. 
\medskip
%

\subsection{Two Previous Results in Three Dimensions}

 \begin{thm}\label{Gruber1}
 {\bf (Gruber's Theorem 2 in  \cite{Gr2}.)}
Let  $(Y_n)_n$ be a sequence of minimizers, i.e. 
$(Y_n)_n$, with $Y_n\in\argmin_{\sh Y=n}E(Y)$. \\
\begin{enumerate}
\item Then for some positive integer $n_0$, if $n >n_0$ 
there exists $\bt>1$ such that $Y_n$ is a $((1/\bt) n^{-1/3},  n^{-1/3})$-Delone set, i.e.,
 $$ n^{-1/3}\ge\min_{y,y'\in Y_n,\ y\neq y'}  |y-y'|   \ge (1/\bt) n^{-1/3}.$$

\item $Y_n$ is uniformly distributed in $Q$, i.e.
\begin{equation*}
\sh(K\cap Y_n)=|K| n + o(n)\qquad \text{as } n\ra +\8 
\end{equation*}
for any Jordan measurable set $K\sse Q$.
\end{enumerate}
\end{thm}

\medskip

 \begin{thm}\label{Zador}
{\bf  (Zador's uniform energy formula in \cite{Zador}, 3D case.)}  
  There exists some constant $\tau>0$ such that given any sequence 
 $Y_n\in\argmin_{\sh Y=n}E(Y)$, we have \[ n^{2/3}E(Y_n)\,\, \to\,\, \tau.\] 
\end{thm}
Zador's result has been extended by Gruber in the general setting of manifolds \cite {Gr2}.
%
However, to our knowledge, no further description of the geometry of Voronoi cells has been proven, nor any explicit
lower bounds on $\tau$. 

\medskip

\subsection{The Statement of Our Results}

 For the remainder of this article we assume the space dimension $N=3$. 

\begin{thm}\label{main1}
Let  $n \in \mathbb{N}$ and $Y_n$ be a  minimizer of  $(\ref{Vor-energy})$. Then for any $y\in Y_n$, with $V$ denoting its Voronoi cell, we have:\\
(i) There exists constants $\ggm_1, \dots, \ggm_5$ (independent of $n$)  such that 
 \begin{align}
\diam(V)&\ge \ggm_3n^{-1/3}, \label{cor-diam-low}\\
|V|&\ge \omega_3 \ggm_5^3n^{-1},\label{cor-vol-low}\\
\diam(V)&\le \ggm_4(n-2)^{-1/3},\label{cor-diam-high}\\
V \text{ has at most }  N_\ast &:= 2(3\ggm_4/\ggm_5)^3 \text{ faces}\label{cor-f},
\end{align}
where $\omega_3:=4\pi/3$, and
\begin{align*}
\ggm_1&:=(2/5)^{2/3}/40\approx 0.013572,
\qquad \ggm_3:=\omega_3^{-1/5}\ggm_1^{1/5}\approx 0.317769,\\
\ggm_5&:=\frac14\bigg(\sqrt{1+\frac{2^4\cdot 3^3}{5^2\cdot 10^3}}-1\bigg)\ggm_3 \approx 0.000451,\\
\ggm_4&:= \frac{2\cdot 12^{1/4}(16)^{1/3} }{\pi^{1/4}\omega_3^{1/12}}
\bigg(\sqrt{1+\frac{2^4\cdot 3^3}{5^2\cdot 10^3}}-1\bigg)^{-1/2}\bigg(\frac{5^2\cdot 10^3}{2^2\cdot 3^3}\bigg)^{1/4} \approx 333.18 \qquad \qquad  N_\ast \approx 2.94 \times 10^{20}.
\end{align*}

(ii) Let $\tau$ be the constant in  Zador's asymptotic estimate (cf. Theorem $\ref{Zador}$), that is,  
  \[ n^{2/3}E(Y_n)\to\tau.\]
  Then we have 
  \begin{equation}\label{en}
n^{2/3}E(Y_n)\ge \tau \qquad \fal n\gg1
\end{equation}
with 
\begin{equation}\label{tau-est}
\tau\ge \frac{2\pi}{5}\omega_3^{-5/3}\approx 0.11545.
\end{equation}
%
%
\end{thm}

\bigskip
The lower bound on $\tau$ given in \eqref{tau-est} is approximately half the energy density of
the BCC lattice ($\approx 0.23562$),
the conjectured asymptotically optimal configuration.
%
%
%
%
%
%
%
%
%
%
%
The proofs of the statements comprising Theorem \ref{main1} are presented in Sections \ref{proofs1}-\ref{S3}. \\

{\bf Remark:} While we state and proof Theorem \ref{main1} in three dimensions, our proofs work in {\it any} space dimension, with appropriate adjustments for the constants. 

\section{The proof of  Theorem \ref{main1}{\it (i)}}\label{proofs1}

In this section we prove the statements \eqref{cor-diam-low}--\eqref{cor-f} of Theorem \ref{main1}, 
in the exact same order they are stated. Their proofs will rely on the following two lemmas whose proofs are presented later in Section \ref{S3}. 

\begin{lem}\label{d}
Given a compact, convex set $V\sse \R^3$, 
a point $y$ in the interior of $V$, then there exists $y'\in V$ such that
\begin{equation}
\int_{V} [|x-y|^2 - d^2(x,\{y,y'\})]\d x \ge \max\bigg\{\frac{2^2\cdot 3^3}{5^2\cdot 10^3}r^2|V|, \ggm_1|V|^{5/3}\bigg\},
\label{1.5}
\end{equation}
where $r:=\max_{z'\in \pd V}|z'-y|$, $\ggm_1=(2/5)^{2/3}/40$.
\end{lem}

\begin{lem}\label{below}
(Lower bound on the distance to a closest neighbor) Given $n$, let $Y_n$ be a minimizer. Then for any $y\in Y_n$ with $V$ denoting its Voronoi cell,  we have \begin{align*}
\min_{z\in Y_n \backslash\{ y\}} |y-z|&\ge 
 r\bigg(\sqrt{1+\frac{2^4\cdot 3^3}{5^2\cdot 10^3}}-1\bigg)\ge 
 \ggm_2|V|^{1/3},
 \end{align*}
 where 
\[ \ggm_2:= \bigg(\sqrt{1+\frac{2^4\cdot 3^3}{5^2\cdot 10^3}}-1\bigg)
 \omega_3^{-1/3} \qquad {\rm and} \qquad r:=\max_{z'\in \pd V}|z'-y|.
\] 
\end{lem}

\bigskip

\subsection{Lower bound on the diameter: proof of statement \eqref{cor-diam-low}}

To proof of  statement \eqref{cor-diam-low} of Theorem \ref{main1} will only require Lemma \ref{d}.

\begin{proof}({\bf of statement} \eqref{cor-diam-low})
Let $s:=\diam(V)$.
We claim:
\begin{equation}\label{claim 1}
 \text{there exists } y'\in Y_n\backslash \{y\} \text{ such that } |y'-y|\le 2s.
\end{equation}
  The proof is by contradiction:
assume the opposite, i.e. there are no other points of $Y_n\backslash \{y\}$ in the ball $B(y,2(s+\vep))$ for some $\vep>0$. Then
let $z$ be an arbitrary point with $|z-y|=s+\vep/2$: clearly $z\in V$, as the opposite would give the existence
of $y'\in Y_n$ with $|z-y'|\le |z-y|=s+\vep/2$, hence
$$2(s+\vep)\le |y'-y|\le |z-y'|+|z-y|\le 2s+\vep,$$
which is a contradiction. Thus any such $z$ satisfying $|z-y|=s+\vep/2$ belongs to $V$, hence
$B(y,s+\vep/2)\sse V$, contradicting $\diam(V)=s$, and \eqref{claim 1} is proven.
 
Let $y'\in Y_n\backslash \{y\}$ be a point satisfying $|y'-y|\le 2s$. If we remove $y$, then all points of $V$ can still project on $y'$,
in the sense that for any $x\in V$ we have
\begin{align*}
|x-y'|^2-|x-y|^2 &=(|x-y'|-|x-y|)(|x-y'|+|x-y|) \le |y-y'|(2s+|y-y'|) \le 8s^2.
\end{align*}
 Integrating over $V$  yields
\begin{equation*}
\int_V [|x-y'|^2-|x-y|^2]  \d x\le 8s^2|V|.
\end{equation*}
Since $\diam (V)=s$, it follows that $V$ is contained in a ball of diameter $s$, hence
\begin{equation*}
\int_V [|x-y'|^2-|x-y|^2]  \d x\le 8s^2|V| \le \omega_3 s^5.
\end{equation*}
Thus by removing $y$, the energy increases by at most $\omega_3 s^5$. 
The average volume of all Voronoi cells is $n^{-1}$, thus there exists $y'$ whose Voronoi cell $V'$
has volume at least $n^{-1}$. Lemma \ref{d} gives that it is possible to add $\~y'$ in $V'$, and the energy
is decreased by at least
$\ggm_1 n^{-5/3}$. 
By the minimality of $Y_n$ we get
\begin{equation*}
\omega_3 s^5 \ge \ggm_1 n^{-5/3}\Lra s\ge \ggm_3n^{-1/3},\qquad \ggm_3=\omega_3^{-1/5}\ggm_1^{1/5},
\end{equation*}
concluding the proof.
\end{proof}

\subsection{Lower bound on the volume: proof of statement \eqref{cor-vol-low}}

The  proof of  \eqref{cor-vol-low} only requires Lemma \ref{below}. 

\begin{proof}({\bf of statement} \eqref{cor-vol-low})
Consider an arbitrary $y\in Y_n$, and denote by $V$ its Voronoi cell. Set $r:=\max_{z'\in \pd V}|z'-y|$,
and for any pair $z_1,z_2\in V$ such that $|z_1-z_2|=\diam(V)$, we have
\begin{equation}\label{diam r}
\diam(V) = |z_1-z_2|\le |z_1-y|+|y-z_2|\le 2r\Lra r\ge \diam(V)/2. 
\end{equation}
Choose  $y'\in Y_n\backslash \{y\}$ such that $|y-y'|=\min_{z\in Y_n\backslash \{y\}}|y-z|$, and
by Lemma \ref{below}, \eqref{diam r} and \eqref{cor-diam-low} we have
\begin{eqnarray}
|y-y'| &\overset{\text{Lemma \ref{below}}}\ge& r\bigg(\sqrt{1+\frac{2^4\cdot 3^3}{5^2\cdot 10^3}}-1\bigg) \notag \\
& \overset{\eqref{diam r}}\ge &
\frac12\bigg(\sqrt{1+\frac{2^4\cdot 3^3}{5^2\cdot 10^3}}-1\bigg)\diam(V)\notag\\
&\overset{\eqref{cor-diam-low}}\ge& \frac12\bigg(\sqrt{1+\frac{2^4\cdot 3^3}{5^2\cdot 10^3}}-1\bigg) \ggm_3n^{-1/3}.\label{est 1}
\end{eqnarray}
Using the same arguments from the proof of \eqref{claim 1}, we now prove that 
\begin{equation}
B(y, \ggm_5 n^{-1/3})\sse V\qquad {\rm where} \qquad \ggm_5=\frac14\bigg(\sqrt{1+\frac{2^4\cdot 3^3}{5^2\cdot 10^3}}-1\bigg)\ggm_3.
\label{contain}
\end{equation}
To this end, assume the opposite, i.e. there exists some $z\in B(y, \ggm_5 n^{-1/3})$ with $z\notin V$. Thus
there exists $y''\in Y_n \backslash \{y\}$ such that $|z-y''|<\ggm_5 n^{-1/3}$. Thus
\begin{align*}
|y''-y|\le |z-y''|+|z-y|<2\ggm_5 n^{-1/3}\overset{\eqref{est 1}}\le |y-y'|,
\end{align*}
which contradicts $|y-y'|=\min_{z\in Y_n\backslash \{y\}}|y-z|$. Thus \eqref{contain} is proven, which in turn
gives
\begin{equation*}
\omega_3\ggm_5^3n^{-1}=|B(y, \ggm_5 n^{-1/3})|\le | V|,
\end{equation*}
hence \eqref{cor-vol-low}.
\end{proof}

\subsection{Upper bound on the diameter: proof of statement \eqref{cor-diam-high}}

The proof of   \eqref{cor-diam-high}  requires both Lemma \ref{d} and Lemma \ref{below}. 

\begin{proof}({\bf of statement} \eqref{cor-diam-high})
Upon renaming, let $y_1$ be such that its Voronoi cell $V_1$ has maximum diameter. Let $r_1:=\max_{z'\in \pd V_1}|z'-y_1|$,
and note that denoting by $w,u\in V_1$ two points realizing the diameter, we have 
 \begin{equation*}
|w-u|=\diam(V_1)\le|w-y_1|+|u-y_1|\le 2r_1.
\end{equation*}
Next we prove the existence of a cell $V_2$, with generator $y_2$, such that
\begin{equation}\label{cell2}
|V_2|\le \frac{2}{(n-2)}\qquad{\rm and} \qquad  \quad  \sigma(y_2)^3\le \frac{16}{\omega_3(n-2)} \quad {\rm where} \quad  \sigma(y_2):=\min_{z\in Y_n\backslash \{y_2\}}|y_2-z|.
\end{equation}
To this end, we note the following. 
\begin{enumerate}
\item[(a)] Denoting by
\begin{equation*}
\mathcal{V}_n:=\{y\in Y_n: \text{ the Voronoi cell } V_y\text{ of } y \text{ satisfies } |V_y|\ge 2/(n-2)\},
\end{equation*}
we claim $\sh\mathcal{V}_n\le \lceil n/2\rceil$.
This is because 
the total number of cells is $n$, and if the opposite holds, i.e.
if there exists at least $n-\lceil n/2\rceil \ge (n-1)/2$ cells with volume greater than $2/(n-2)$, 
we conclude that 
 \begin{equation*}
1=|Q|\ge\sum_{y\notin \mathcal{V}_n} |V_y|\ge (n-\lceil n/2\rceil)\frac2{n-2} \ge \frac{n-1}{2}\frac2{n-2}>1.
\end{equation*}

\item[(b)] Similarly if we denote by
\begin{equation*}
\mathcal{S}_n:=\left\{y\in Y_n: \sigma(y)^3\le \frac{16}{\omega_3(n-2)}\right\},\qquad \sigma(y):=\min_{z\in Y_n\backslash \{y\}}|y-z|,
\end{equation*}
we claim $\sh \mathcal{S}_n\ge \lceil n/2\rceil+1$. To this end,  for any $y$  
we have $B(y,\sigma(y)/2)\sse V_y$, and 
hence 
$|V_y|\ge \omega_3\sigma(y)^3/8$. If by contraction we had $\sh \mathcal{S}_n\le \lceil n/2\rceil$, i.e. 
there exist at least $n-\lceil n/2\rceil$ generators $y$ with $\sigma(y)^3\ge \frac{16}{\omega_3(n-2)}$, we would conclude that 
\begin{align*}
1&=|Q|\ge\sum_{y\notin \mathcal{S}_n} |V_y| \ge \frac{\omega_3}8\sum_{y\notin \mathcal{S}_n}\sigma(y)^3
\ge \frac{\omega_3}8\frac{n-1}2 \frac{16}{\omega_3(n-2)}>1.
\end{align*}
\end{enumerate}
Combining (a) and (b) above yields the existence of a cell $V_2$ with generator $y_2$   satisfying \eqref{cell2}.

Next,  we estimate how much the total energy increases if we remove $y_2$. 
Let $y_3$ be such that
$|y_2-y_3|=\sigma(y_2)$.
Then for any $x\in V_2$, we have 
\begin{align*}
|x-y_3|^2-|x-y_2|^2&\le |y_2-y_3|(|x-y_3|+|x-y_2|) \le \sigma(y_2)(2|x-y_2|+\sigma(y_2)). 
\end{align*}
Noting that the midpoint $\-y:=(y_2+y_3)/2\in \pd V_2$, we have 
\begin{equation*}
\sigma(y_2) = 2|y_2-\-y|\le 2\diam(V_2)\le 2\diam(V_1).
\end{equation*}
Thus we have 
\begin{align}
|x-y_3|^2&-|x-y_2|^2 \le 4\diam(V_1)\sigma(y_2) \notag
\end{align}
which implies 
\begin{eqnarray}
 \int_{V_2}(|x-y_3|^2-|x-y_2|^2)\d x &\le& 4\diam(V_1)\sigma(y_2)|V_2| \notag\\
&\overset{(\ref{cell2})}\le & 
\frac{8\cdot(16)^{1/3}\diam(V_1) }{\omega_3^{1/3}(n-2)^{4/3}} \notag\\
&\le& \frac{(16)^{4/3}r_1 }{\omega_3^{1/3}(n-2)^{4/3}}.\label{enl}
\end{eqnarray}

By Lemma \ref{d}, we can always add a point in $V_1$ and the energy is decreased by at least
$\frac{2^2\cdot 3^3}{5^2\cdot 10^3}r_1^2|V_1|$. Hence  we need to bound $|V_1|$ from below.  To this end, choose an arbitrary $z_1$ such that $|z_1-y_1|=r_1$,
and let $\ell$ be the line through $y_1$ and $z_1$, and let $\Pi$ be the plane through $y_1$ and orthogonal to $\ell$.
By Lemma \ref{below},  we have 
\begin{equation*}
\sigma(y_1):=\min_{z\in Y_n\backslash \{y_1\}}|y_1-z|\ge r_1\bigg(\sqrt{1+\frac{2^4\cdot 3^3}{5^2\cdot 10^3}}-1\bigg),
\end{equation*}
and hence $B(y_1,\sigma(y_1)/2)\sse V_1$. In particular, by convexity of $V_1$, the disk $\Pi\cap B(y_1,\sigma(y_1)/2)\sse V_1$,
and the cone with base $\Pi\cap B(y_1,\sigma(y_1)/2)$ and height $\{(1-s)y_1+sz_1:s\in [0,1]\}$
is again contained in $V_1$. It follows that 
\begin{equation*}
|V_1|\ge r_1\frac{\pi \sigma(y_1)^2}{12} \ge r_1^3 \frac\pi{12}\bigg(\sqrt{1+\frac{2^4\cdot 3^3}{5^2\cdot 10^3}}-1\bigg)^2.
\end{equation*}
Consequently, there exists $y'\in V_1$ such that
\begin{eqnarray*}
\int_{V_1}(|x-y_1|^2-d^2(x,\{y_1,y'\}))\d x &\overset{{\rm Lemma \, \ref{d}}}\ge& \frac{2^2\cdot 3^3}{5^2\cdot 10^3}r_1^2|V_1| \\
& \ge & \frac\pi{12}\bigg(\sqrt{1+\frac{2^4\cdot 3^3}{5^2\cdot 10^3}}-1\bigg)^2\frac{2^2\cdot 3^3}{5^2\cdot 10^3}r_1^5.
\end{eqnarray*}
Combining with \eqref{enl} and using the minimality of $Y_n$, we infer
\begin{align*}
 \frac{(16)^{4/3}r_1 }{\omega_3^{1/3}(n-2)^{4/3}} &\ge\frac\pi{12}
\bigg(\sqrt{1+\frac{2^4\cdot 3^3}{5^2\cdot 10^3}}-1\bigg)^2\frac{2^2\cdot 3^3}{5^2\cdot 10^3}r_1^5\\
&\Lra r_1^4\le  \frac{12(16)^{4/3} }{\pi\omega_3^{1/3}(n-2)^{4/3}}
\bigg(\sqrt{1+\frac{2^4\cdot 3^3}{5^2\cdot 10^3}}-1\bigg)^{-2}\frac{5^2\cdot 10^3}{2^2\cdot 3^3}\\
&\Lra \diam(V_1)\le 2r_1\le 
\frac{2\cdot 12^{1/4}(16)^{1/3} }{\pi^{1/4}\omega_3^{1/12}(n-2)^{1/3}}
\bigg(\sqrt{1+\frac{2^4\cdot 3^3}{5^2\cdot 10^3}}-1\bigg)^{-1/2}\bigg(\frac{5^2\cdot 10^3}{2^2\cdot 3^3}\bigg)^{1/4},
\end{align*}
concluding the proof. 
\end{proof}


%

\subsection{Upper bound on the number of faces: Proof of statement \eqref{cor-f}}
The bounds on the diameter (statement \eqref{cor-diam-high}) and volume (statement \eqref{cor-vol-low}) of Voronoi cells allow us to bound their geometric complexity (i.e. the maximum number of faces).

\begin{proof}({\bf of statement \eqref{cor-f}})
Consider an arbitrary $y\in V$. By construction, its Voronoi cell $V$ is the bounded convex region delimited by the axial planes (i.e.
the plane orthogonal to the line segment and passing through its midpoint)
of the line segments connecting $y$ and some other generator $y'\in Y_n$. 

Statement \eqref{cor-diam-high} implies that any Voronoi cell has diameter not exceeding $\ggm_4{(n-2)}^{-1/3}$. Thus
if two generators $y',y''\in Y_n$ satisfy $|y'-y''|>2\ggm_4{ (n-2)}^{-1/3}$, then their Voronoi cells do not share boundaries.
Thus only the generators $y'\in B(y,2\ggm_4{ (n-2)}^{-1/3})$ can have their Voronoi region share  a boundary with $V$.
Again, the upper bound on the diameter given by estimate \eqref{cor-diam-high} gives that any Voronoi cell (of any generator
$y'\in B(y,2\ggm_4{ (n-2)}^{-1/3})$) is entirely contained in $B(y,3\ggm_4{ (n-2)}^{-1/3})$.

Statement \eqref{cor-vol-low} implies that each Voronoi cell has volume at least $\omega_3 \ggm_5^3n^{-1}$,
so the ball \[ B(y,3\ggm_4{ (n-2)}^{-1/3})\]  can contain only
\begin{equation*}
\frac{\omega_3  (3\ggm_4{ (n-2)}^{-1/3})^3 }{\omega_3 \ggm_5^3 n^{-1}} = (3\ggm_4/\ggm_5)^3
{\frac{n}{n-2} \le 2 (3\ggm_4/\ggm_5)^3=:N_\ast }
\end{equation*}
{\em whole} Voronoi cells. { The last factor 2 comes from the fact that any
polyhedron has at least 4 faces, and $n/(n-2)\le 2$ for all $n\ge 4$.}
Thus $V$ can share boundary with at most $N_\ast$ other Voronoi cells.
\end{proof}

\section{Energy estimates: proof of Theorem \ref{main1}{\it (ii)}}\label{proofs2}

\subsection{Proof of \eqref{en}}

\begin{proof} ({\bf of estimate \eqref{en}})
Recall first that $Q=[0,1]^3$.
Let $\tau$ be the energy of the ground state, i.e.
\begin{equation}\label{tau}
\tau:=\liminf_{n\to+\8}\inf_{\sh Y=n} n^{2/3}E(Y).
\end{equation}
\begin{itemize}
\item Claim 1: the limit inferior in \eqref{tau} is a true limit.
\end{itemize}
Although this has been proven in \cite{Zador}, we use here an alternative construction that will be crucial
for Claim 2 below. The proof is done by contradiction: assume not, i.e. 
\begin{align*}
\bt:=\limsup_{n\to+\8}\inf_{\sh Y=n} n^{2/3}E(Y)>\tau.
\end{align*}
Consider sequence of minimizers $(Y_n)_n$ realizing this limit superior, i.e.
\begin{equation*}
\bt=\lim_{n\to+\8} n^{2/3}E(Y_n),\qquad \sh Y_n\equiv n,
\end{equation*}
and without loss of generality we assume
\begin{equation*}
(\fal n) \qquad n^{2/3}E(Y_n)\ge \frac{\bt+\tau}{2}+\vep \qquad \text{for some }\vep\in(0,(\bt-\tau)/2).
\end{equation*}
Then we take $Z$ with 
\begin{equation*}
k^{2/3}E(Z)<\frac{\bt+\tau}{2}-\vep, \qquad k:=\sh Z,
\end{equation*}
which surely exists since, due to the definition of $\tau$, there exists a sequence $(Z_n)_n$ such that
$k^{2/3}E(Z)\to \tau$. 

For any $c>0$, let $cZ$ (resp. $cQ$) be the image of $Z$ (resp. $Q$) under the scaling of factor $c$. 
Note that the cube $\frac{\lfloor(n/k)^{1/3} \rfloor}{(n/k)^{1/3}}Q$ can be tessellated
with $\lfloor(n/k)^{1/3} \rfloor^3$ 
identical copies of $(k/n)^{1/3}Z$: this because we can
partition the segment $[0,\frac{\lfloor(n/k)^{1/3} \rfloor}{(n/k)^{1/3}}]$ into $\lfloor(n/k)^{1/3} \rfloor$
intervals of length $(k/n)^{1/3}$.

Let $Y_n'$ be the competitor obtained by tessellating $\frac{\lfloor(n/k)^{1/3} \rfloor}{(n/k)^{1/3}}Q$ with
$\lfloor(n/k)^{1/3} \rfloor^3$
identical copies of $(k/n)^{1/3}Z$: clearly
\begin{equation*}
\sh Y_n' = k \lfloor(n/k)^{1/3} \rfloor^3\le n=\sh Y_n.
\end{equation*}
The minimality of $Y_n$ gives immediately $E(Y_n)\le E(Y_n')$. On the other hand, 
note that any point $z\in Q\backslash \frac{\lfloor(n/k)^{1/3} \rfloor}{(n/k)^{1/3}}Q$ has distance at most $3(k/n)^{1/3}$ from $Y_n'$:
thus
\begin{equation*}
\int_{Q\backslash \frac{\lfloor(n/k)^{1/3} \rfloor}{(n/k)^{1/3}}Q} d^2(z,Y_n')\d z
\le 9(k/n)^{2/3} \Big|Q\backslash \frac{\lfloor(n/k)^{1/3} \rfloor}{(n/k)^{1/3}}Q\Big|
\le 54k/n .
\end{equation*}
Since we have the scaling law 
$E(cZ)=c^5E(Z)$ for any $c>0$, it follows 
\begin{align*}
E(Y_n')&\, =\, 
\int_{\frac{\lfloor(n/k)^{1/3} \rfloor}{(n/k)^{1/3}}Q} d^2(z,Y_n')\d z\, +
\, \int_{Q\backslash \frac{\lfloor(n/k)^{1/3} \rfloor}{(n/k)^{1/3}}Q} d^2(z,Y_n')\d z\notag\\
&\, \le\, \lfloor(n/k)^{1/3} \rfloor^3 E((k/n)^{1/3}Z)+54k/n\notag\\
&\, = \, \lfloor(n/k)^{1/3}\rfloor^3 (k/n)^{5/3} E(Z)+54k/n\notag\\
&\, \le \, \lfloor(n/k)^{1/3}\rfloor^3 (k/n)^{5/3}\Big(\frac{\bt+\tau}{2}-\vep\Big) k^{-2/3}+54k/n\notag\\
&\, \le\,  n^{-2/3}\Big(\frac{\bt+\tau}{2}-\vep\Big)+54k/n\\
&\, <\, n^{-2/3}\Big(\frac{\bt+\tau}{2}+\vep\Big)\,  \le\,  E(Y),
\end{align*}
which is a contradiction. This proves that such a $\bt$ cannot exist, and the limit inferior in \eqref{tau}
is in reality a limit.

\medskip

Now we prove
\begin{itemize}
\item Claim 2: for any $n$ and minimizer $Y_n$ with $n$ generators, it holds $n^{2/3}E(Y_n)\ge \tau$.
\end{itemize}
Assume the opposite, i.e. there exists $k$ and a minimizer $Y_k\in \argmin_{\sh Y'=k}E(Y')$ such that
$k^{2/3}E(Y_k)<\tau$.
Let $k_n:=n^3k$. Divide $Q$ into $n^3$
smaller cubes, each of which congruent to $\frac1n Q$, and 
put in each of such cubes the (scaled) configuration $\frac1n Y_k$.
Let $Y_{k_n}$ be the configuration obtained by stacking $n^3$ such cubes. Recalling that if a Voronoi cell
$V$ is scaled by a factor $c$, the volume scales by a factor $c^3$, and the energy scales by a factor
$c^5$, we obtain
\begin{equation*}
E(Y_{k_n})\le n^3(E(Y_k)n^{-5}),
\end{equation*}
and
\begin{equation*}
k_n^{2/3}E(Y_{k_n})\le n^2k^{2/3}[n^3(E(Y_k)n^{-5})]=k^{2/3}E(Y_k)<\tau
\end{equation*}
for any $n$, contradicting the minimality of $\tau$.
\end{proof}

\subsection{Proof of \eqref{tau-est}}
We will use the following result which was proven in \cite[Lemma~2.5]{BJM}. 
\begin{lem}\label{sphere}
Among all convex sets with fixed volume, the sphere has the lowest energy.
\end{lem}


\begin{proof}({\bf of statement \eqref{tau-est}})
Consider a sequence of minimizers $(Y_n)_n$, and choose an arbitrary element $Y_m$.
 Note that the union $\bigcup_{y\in Y_m}V_y$ has volume 1, where $V_y$ denotes the Voronoi
 cell of $y$. 
 Lemma \ref{sphere} gives that among all convex sets of unit volume, the sphere has the lowest energy,
which is equal to
\begin{equation*}
\frac{2\pi}{5}\Big( \frac3{4\pi}\Big)^{5/3}.
\end{equation*}
Scaling arguments give that as the volume scales by a factor of $s$, the energy scales by a factor of
$s^{5/3}$, hence the energy of a Voronoi cell with volume $|V|$ is at least
\begin{equation*}
\frac{2\pi}{5}\Big( \frac3{4\pi}\Big)^{5/3}|V|^{5/3}.
\end{equation*}
Therefore,
\begin{align*}
E(Y_m) & = \sum_{y\in Y_m} \int_{V_y}|x-y|^2\d x\ge \frac{2\pi}{5}\Big( \frac3{4\pi}\Big)^{5/3}  \sum_{y\in Y_n}|V_y|^{5/3}.
\end{align*}
Using the convexity of $f(t)=t^{5/3}$ and the fact that the average volume is $1/m$, we infer
\begin{equation*}
E(Y_m) \ge \frac{2\pi}{5}\Big( \frac3{4\pi}\Big)^{5/3}  \sum_{y\in Y_n}|V_y|^{5/3}
\ge \frac{2\pi}{5}\Big( \frac3{4\pi}\Big)^{5/3} m^{-2/3},
\end{equation*}
hence
\begin{equation*}
\frac{2\pi}{5}\Big( \frac3{4\pi}\Big)^{5/3} \le m^{2/3}E(Y_m) \to \tau,
\end{equation*}
and the proof is complete.
\end{proof}


\section{Proofs of Lemmas \ref{d} and \ref{below}}\label{S3}

\begin{proof}({\bf of Lemma \ref{d}})
We first prove
\begin{equation}
\int_{V} [|x-y|^2 - d^2(x,\{y,y'\})]\d x \ge \frac{2^2\cdot 3^3}{5^2\cdot 10^3}r^2|V|.\label{t1-d}
\end{equation}
Let $z\in \pd V$ be a point satisfying
\begin{equation*}
|z-y|=\max_{z'\in V}|z'-y|=\max_{z'\in \pd V}|z'-y|,
\end{equation*}
and let $r:=|z-y|$. Endow $\R^3$ with the cartesian system $(x_1,x_2,x_3)$ with
\begin{equation*}
y=(0,0,0),\qquad z=(r,0,0),
\end{equation*}
and we add a point $y'=(x,0,0)\in V$, with fixed $x\in (0,r)$ to be determined shortly. 

Define  $\Pi_x:=\{x_1=x\}$. 
Since $V$ is convex, the intersection
$\Pi_x\cap  V$ is also convex for all $x$, and the boundary $\pd  V\cap \Pi_x$ is a convex Jordan curve.
Let $\gm:[0,1]\lra \pd  V\cap \Pi_x$ be an arbitrary parameterization, and for any $t\in [0,1]$, let
$\ell_t$ be the half-line starting from $z$ passing through $\gm(t)$. 

The convexity of $V$ now has the following geometric consequences:
\begin{enumerate}
\item[(G1)] $V$ surely contains the ``cone'' delimited by the surfaces $V\cap \Pi_x$ and $\bigcup_{t\in[0,1]} (\ell_t\cap \{x_1\ge x\})$,

\item[(G2)] for any $t\in [0,1]$, the half-line $\ell_t$ exits $V$ at $\gm(t)$, that is, $\ell_t\cap \{x_1<x\}=\emptyset$.
\end{enumerate}
Now let $V_+(x):=V\cap \{x_1\ge x\}$,
 and we estimate its volume. By construction, in view of $|z-y|=\max_{z'\in V}|z'-y|$ and observation (G2),
 it follows that $V\cap \{x_1< x\}$ must be contained in the truncated cone (that we denote by $\mathcal{C}_-$)
 delimited by the surfaces $\bigcup_{t\in[0,1]}  (\ell_t\cap \{-r\le x_1\le x\})$, $\{x_1=-r\}$ and $\Pi_x$.
 Let $\mathcal{C}$ be the cone delimited by $\bigcup_{t\in[0,1]}  (\ell_t\cap \{-r\le x_1\})$ and $\{x_1=-r\}$ and $\Pi_x$,
 and note that the cone 
$\mathcal{C}_+:=\mathcal{C}\backslash \mathcal{C}_-$ satisfies
$$\dfrac{|\mathcal{C}_+|}{|\mathcal{C}|}=\left(\dfrac{r-x}{2r}\right)^3\Lra \dfrac{|\mathcal{C}_+|}{|\mathcal{C}_-|} =
\frac{\left(\frac{r-x}{2r}\right)^3}{1-\left(\frac{r-x}{2r}\right)^3}=\frac{(r-x)^3}{8r^3-(r-x)^3}.$$
Since by construction we have
$V\sse \mathcal{C}_- \cup V_+$, and $\mathcal{C}_+ \sse V_+$, it follows
\[|V|\le |\mathcal{C}_- |+| V_+| \qquad  {\rm and} \qquad |V_+|\ge |\mathcal{C}_+ | = \frac{(r-x)^3}{8r^3-(r-x)^3}|\mathcal{C}_- |.\]
Hence, we have 
\[ |\mathcal{C}_- | \le |V_+| \frac{8r^3-(r-x)^3}{(r-x)^3}=|V_+| \bigg[\frac{8r^3}{(r-x)^3}-1\bigg] 
\qquad {\rm and \,\, so} \qquad 
|V|\le |\mathcal{C}_- |+| V_+| \le |V_+| \frac{8r^3}{(r-x)^3}. 
\]
Thus
 \begin{align}
\frac{|V_+|}{|V|} \ge 
\frac{(r-x)^3}{8r^3} \, = \, \left(\frac12-\frac{x}{2r}\right)^3. 
\label{vol ratio}
\end{align}
%
Now take an arbitrary point $w=(w_1,w_2,w_3)\in V_+$ (hence $w_1\in[x,r]$), and note that
\begin{align}
|w-y|^2&=w_1^2+w_2^2+w_3^2, \qquad |w-y'|^2=(w_1-x)^2+w_2^2+w_3^2\notag\\
&\Lra
|w-y|^2-|w-y'|^2 =w_1^2-(w_1-x)^2=x(2w_1-x) \ge x^2.\label{sqdist}
\end{align}
Thus 
\begin{align*}
\int_V[|w-y|^2-d^2(w,\{y,y'\})]\d w  &\ge \int_{V_+}[|w-y|^2-d^2(w,\{y,y'\})]\d w \\
&=\int_{V_+}[|w-y|^2-|w-y'|^2]\d w \\
& \overset{\eqref{sqdist}}\geq |V_+|x^2 \\& \overset{\eqref{vol ratio}}\ge |V|\left(\frac12-\frac{x}{2r}\right)^3 x^2.
\end{align*}
Since the above argument is valid for all $x\in [0,r]$, it follows
\begin{equation*}
\int_V[|w-y|^2-d^2(w,\{y,y'\})]\d w  \ge |V|\max_{x\in [0,r]}\left(\frac12-\frac{x}{2r}\right)^3 x^2.
\end{equation*}
Maximizing the last expression over $x \in [0,r]$ (i.e. taking $x = \frac{2r}{5}$) yields \eqref{t1-d}.

\bigskip
We now prove
\begin{equation}
\int_{V} [|x-y|^2 - d^2(x,\{y,y'\})]\d x \ge \ggm_1|V|^{5/3}.\label{t2-d}
\end{equation}
As in the proof of \eqref{t1-d}, endow $\R^3$ with a Cartesian coordinate system with origin in $y$.
For any $t\in [0, |V|^{1/3}]$, set 
\begin{align*}
Q(t)&:=\{-t/2\le x_1,x_2,x_3 \le t/2\},\qquad
V_k^\pm(t):=V\cap \{\pm x_k\ge t/2\},\quad k=1,2,3.
\end{align*}
Note that since 
\begin{equation*}
V\backslash Q(t)=\bigcup_{k=1}^3 V_k^\pm(t),
\end{equation*}
we have 
\[   |V\backslash Q(t)|=\bigg|\bigcup_{k=1}^3 V_k^\pm(t)\bigg| \, \ge \, |V|-t^3.\]
Thus there exists an element $\~V(t)\in \{ V_k^\pm(t):k=1,2,3\}$ such that $|\~V(t)|\ge (|V|-t^3)/6$. 
Let $y'$ be the center of the face $\~V(t)\cap Q(t)$.
By
\eqref{sqdist}, any $w\in \~V(t)$ satisfies 
$|w-y|^2-|w-y'|^2 \ge t^2/4$, hence 
\begin{align*}
\int_V[|w-y|^2-d^2(w,\{y,y'\})]\d w  &\ge \int_{\~V(t)}[|w-y|^2-d^2(w,\{y,y'\})]\d w \\
&\ge \frac{|\~V(t)|t^2}4\\
&\ge \frac{(|V|-t^3) t^2}{24}.
\end{align*}
This last inequality holds for all $t\in [0, |V|^{1/3}]$. In particular, direct computation gives
that the maximum of $(|V|-t^3) t^2$ is attained at $t^3=2|V|/5$, thus
\begin{equation*}
\int_V[|w-y|^2-d^2(w,\{y,y'\})]\d w  \ge \frac{(|V|-t^3)t^2}{24}\bigg|_{t^3=2|V|/5}= \frac1{40}\bigg(\frac23\bigg)^{2/5} |V|^{5/3}
\end{equation*}
which proves \eqref{t2-d}.

\end{proof}


\begin{proof}({\bf of Lemma \ref{below}})
Although a similar estimate has been proven by Gruber in \cite{Gr2}, the lower bound therein was
only implicit. Here we give an explicit lower bound.
To this end, assume $|V|>0$, otherwise the thesis is trivial. The main idea of the proof is:
\begin{enumerate}

\item first we show that if $Y_n$ is optimal, then $y$ is in the interior of $V$,

\item  then we add another point $y'$ in $V$ (the energy difference is estimated
using Lemma \ref{d}),

\item finally we remove $y$ (energy difference to be estimated by direct computation). 
\end{enumerate}

\medskip

{\em Step 1.} 
Assume by contradiction $y\in \pd V$. Then there exists a plane $\Pi$ such that $V$ is entirely on one side of $\Pi$.
Endow $\R^3$ with a cartesian system with $\Pi=\{(x_1,x_2,x_3):x_1=0\}$, $V\sse \{x_1\ge 0\}$, $y=(0,0,0)$.
Then, 
\begin{equation*}
\int_V |z-y|^2\d z =\int_V [z_1^2+z_2^2+z_3^2]\d z_1\d z_2\d z_3,
\end{equation*}
with $z_1\geq0$ for all $z\in V$. Therefore,
\begin{equation*}
\frac{\pd}{\pd y_1}\int_V [(z_1-y_1)^2+z_2^2+z_3^2]\d z_1\d z_2\d z_3\bigg|_{y_1=0} = -2\int_V z_1\d z_1\d z_2\d z_3<0,
\end{equation*}
and $Y_n$ cannot be a minimizer.

\medskip

{\em Step 2.} In Step 1 we have proven that $y$ must be in the interior of $V$, thus we are under the hypotheses of
Lemma \ref{d}, which gives that there exists $y'$ such that 
\begin{equation}
\int_{V} [|x-y|^2 - d^2(x,\{y,y'\})]\d x \ge |V|\frac{2^2\cdot 3^3}{5^2\cdot 10^3}r^2
\ge \ggm_1|V|^{5/3},\qquad
 r:=\max_{z'\in \pd V}|y-z'|.
\label{en3-1}
\end{equation}
This means that adding $y'$ in $V$, the energy decreases by at least $ |V|\frac{2^2\cdot 3^3}{5^2\cdot 10^3}r^2$.

\medskip

{\em Step 3.} Now we have to remove $y$, and estimate how much the energy increases. Set 
$$s:=|y-y''|=\min_{z\in Y_n,\ z\neq y}|y-z|,$$
and for any $x\in V$ it holds
\begin{align}
|x-y''|^2-|x-y|^2&=(|x-y''|-|x-y|)(|x-y''|+|x-y|)\notag\\
&\le |y-y''|(2|x-y|+|y-y''|) \le s(2r+s)\notag\\
&\Lra \int_V [|x-y''|^2-|x-y|^2]\d x \le |V|s(2r+s).\label{en3-2}
\end{align}
Combining \eqref{en3-1}, \eqref{en3-2} and the minimality of $Y_n$ then gives
\begin{equation*}
s^2+2rs- r^2\frac{2^2\cdot 3^3}{5^2\cdot 10^3} \ge 0 \Lra s\geq r \bigg(\sqrt{1+\frac{2^4\cdot 3^3}{5^2\cdot 10^3}}-1\bigg).
\end{equation*}
Finally note that $V\sse B(y,r)$, hence $\omega_3 r^3\ge |V|$, and
$$s\geq r\bigg(\sqrt{1+\frac{2^4\cdot 3^3}{5^2\cdot 10^3}}-1\bigg)\ge \bigg(\sqrt{1+\frac{2^4\cdot 3^3}{5^2\cdot 10^3}}-1\bigg)
 \omega_3^{-1/3}|V|^{1/3},$$
and the proof is complete.
\end{proof}
\section{Towards a Proof of Gersho's conjecture in 3D}\label{sec-Gersho}

Let us now address the extension to 3D of Gruber's 2D proof of Gersho's conjecture. 
The following analogous results are needed: 

\begin{enumerate}
\item we first note that the average number of faces (as $n\to+\8$) of Voronoi cells is some number $\overline{m}\le 14$. This has been 
proven in \cite{14}. Note that 
$14$ is the number of faces of truncated octahedron (Voronoi cells in the BCC lattice). 

\item we {\bf need to verify} that the function
\begin{equation*}
m \quad \longmapsto \quad \min_{\substack{V \text{ convex polytope, } |V|=\al\\
V \text{ has at most } m \text{ faces }  }} \,\,\, \,\,\int_V |x-y|^2\d x,\qquad y  \text{ centroid of } V,
\end{equation*}
is convex for $m\leq N_\ast$,  where $N_\ast$ is given by Theorem \ref{main1}.
 This will allow us to  extend this function (denoted below by $G$) to the continuum $m\in [0,N^*]$, so we
can then verify its convexity. This step ensures that we can then compute
the Hessian matrix of $G$. 
\item we note that the optimal polytope $V$ with $m$ faces is space tiling.
\item we can dispense with the energetic contributions of the boundary cell. 
\end{enumerate}

With these results in hand, 
Gruber's method would then be as follows: 
let 
\begin{equation*}
G(a,m)\,:=\, \min_{\substack{V \text{ convex polytope, } |V|=a\\
V \text{ has at most } m \text{ faces }  }}\,\,\,\,  \int_V |x-y|^2\d x,\qquad y  \text{ centroid of } V,
\end{equation*}
and $G$ is convex in both variables. Then, for any arbitrary tessellation $Y_n$ (with $\sh Y_n=n$), of $Q$, 
let
$\{V_k\}$ be the collection of Voronoi cells, and let $\al_k $ be the number of faces of $V_k$. Then it follows that 
\begin{align*}
E(Y_n)&=\sum_{k=1}^{n}  \int_{V_k} |x-y|^2\d x \\& \ge \sum_{k=1}^{n} G(|V_k|,\al_k)\\
&\ge nG(1/n,\overline{m}) + 
\text{error due to boundary effects}\\
&\ge nG(1/n,14) + 
\text{error due to boundary effects}.
\end{align*}
Since the error due to boundary effects is a higher order term (actually of order
$O(n^{-1})$, compared to $n G(1/n,m)$, which has order $O(n^{-2/3})$, as $n\to+\8$)
it follows that the optimal tessellation (as $n\to+\8$) 
consists of congruent copies of a space tiling polyhedron realizing $G(1/n,14)$. 

\bigskip

Concerning issue (3), we expect the optimal polytope to be the regular truncated octahedron, since:
{ 
\begin{itemize}
 \item it is the tessellation
corresponding to the BCC lattice, which has been proven to be pretty optimal from numerical simulations (see \cite{DW}),

\item it is the {\em only} convex polytope to tile the space by translation, with 14 faces (see \cite[pp.~471--473]{Gr3}). Although
this property is valid for some irregular truncated octahedra too, we expect that for any fixed volume constraint $a$, irregular truncated octahedra
should not realize the minimum in $G(a,14)$.
\end{itemize}
}
Moreover, since a periodic CVT should have generators distributed on a lattice, by
\cite{BS} such a lattice should be the BCC one. 
However, a priori Gersho's conjecture requires only the existence of such a unique ``seed'' polytope for
Voronoi cells, without
any geometric description.

\bigskip

For issue (4), we have the following proposition  which proves that, given any cube $\om\sse Q$, the energy contribution of Voronoi cells intersecting $\pd \om$
 is negligible compared to the 
energy contribution of Voronoi cells not intersecting $\pd \om$. 

\begin{prop}\label{nob}
For any $n$, let $Y_n$ be a minimizer with $\sh Y_n=n$. Then let $\om\sse Q$ be an arbitrary cube
with positive volume, then for any sufficiently large $n$ it holds:
\begin{enumerate}
\item  the contribution to the energy of Voronoi cells intersecting
$\pd \om$ is of order $O(n^{-1})$,

\item  the contribution to the energy of Voronoi cells in $\om$ but not intersecting
$\pd \om$ is of order $O(n^{-2/3})$.
\end{enumerate}
Consequently,
the energy contribution of Voronoi cells intersecting $\pd \om$
 is negligible compared to the 
energy contribution of Voronoi cells in $\om$ not intersecting $\pd \om$. 
\end{prop}

\begin{proof}
Choose $n\gg 1$, and a minimizer $Y_n$ with $\sh Y_n=n$. 
We will establish (from Claims 1--3) that the energy contribution of Voronoi cells
intersecting $\pd \om$ is negligible as $n\to +\8$. In the following $\ggm_i$ ($i=1,3,4$) will be constants
from Theorem \ref{main1}.

\medskip

\begin{itemize}
\item Claim 1: at most $\frac{6\ggm_4}{\omega_3 \ggm_5^3 } n^{2/3}$ Voronoi cells can intersect $\pd \om$.
\end{itemize}
To prove this claim,
estimate \eqref{cor-diam-high} gives that the diameter of each Voronoi cell
is at most $\ggm_4n^{-1/3}$, hence all the Voronoi cells intersecting $\pd \om$ are contained in 
\begin{equation*}
\{x:d(x,\pd \om)\le \ggm_4n^{-1/3}\}.
\end{equation*}
Estimate \eqref{cor-vol-low} gives that the volume of any Voronoi cell is at least $\omega_3 \ggm_5^3 n^{-1}$,
hence
at most
\begin{equation*}
\frac{|\{x:d(x,\pd K)\le \ggm_4n^{-1/3}\}|}{\omega_3 \ggm_5^3 n^{-1}}
\le \frac{6\ggm_4}{\omega_3 \ggm_5^3 }n^{2/3}
\end{equation*}
can intersect $\pd \om$. Thus Claim 1 is proven.

\medskip

\begin{itemize}
\item Claim 2: the energy contribution of all Voronoi cells intersecting $\pd \om$
is at most 
$$\frac{3\ggm_4^6 n^{-5/3}}{4 \ggm_5^3}=O(n^{-1}).$$

\end{itemize}
Let $(y_k)_k\sse Y_n$ be the (finite) collection of atoms such that their Voronoi cells
$(V_k)_k$ intersect $\pd\om$.
Estimate \eqref{cor-diam-high} proves that, for any $k$, $\diam(V_k)\le \ggm_4n^{-1/3}$,
hence $V_k\sse B(y_k,\ggm_4n^{-1/3}/2)$ and
\begin{equation*}
\int_{V_k} |x-y_k|^2\d x \le |V_k|\diam^2(V_k) \le |B(y_k,\ggm_4n^{-1/3}/2)|\diam^2(V_k)
\le \frac{\omega_3\ggm_4^5 n^{-5/3}}8. 
\end{equation*}
Since Claim 1 proves that at most $\frac{6\ggm_4^3}{\ggm_1}n^{2/3}$
Voronoi cells can intersect $\pd\om $, the energy contribution of all such cells is at most
\begin{equation*}
\frac{6\ggm_4}{\omega_3 \ggm_5^3 }n^{2/3}\cdot \frac{\omega_3\ggm_4^5 n^{-5/3}}8 
= \frac{3\ggm_4^6 n^{-1}}{4 \ggm_5^3}
\end{equation*}
and Claim 2 is proven.

\medskip

\begin{itemize}
\item Claim 3: the energy contribution of all Voronoi cells in $\om$ which do not intersect $\pd \om$
is at least 
\begin{equation*}
 \frac{\pi}{5}\omega_3^{-5/3}n^{-2/3}-\frac{3\ggm_4^6 n^{-1}}{4 \ggm_5^3}= O(n^{-2/3}).
\end{equation*}
\end{itemize}
Zador's asymptotic estimate proved that there exists $\tau>0$ such that $n^{2/3}E(Y_n)\to \tau$. 
Thus, for $n$ large we have
\begin{equation*}
2\tau n^{-2/3}\ge E(Y_n)\ge \frac\tau2 n^{-2/3},
\end{equation*}
and the contribution of cells not intersection $\pd \om$ is estimated by
\begin{equation*}
2\tau n^{-2/3}\ge E(Y_n)-\frac{3\ggm_4^6 n^{-5/3}}{4 \ggm_5^3} \ge 
\frac\tau2 n^{-2/3}-\frac{3\ggm_4^6 n^{-1}}{4 \ggm_5^3},
\end{equation*}
and since we proved $\tau\ge \frac{2\pi}{5}\omega_3^{-5/3}$,
 Claim 3 follows.
\end{proof}

\bigskip

Thus the fundamental remaining issue for the proof of Gersho's conjecture in 3D is  (2). 
Note that the convexity of $G$ in the volume variable is almost trivial due to scaling:
without loss of generality assume the centroid is $y=0$, and by using a scaling of ratio $r$ we obtain 
\[\int_{rV} |x|^2 \d x= r^5\int_{V} |x|^2 \d x, \]
{\em independently of } the number of faces of $V$.

{ To prove the convexity of $G$ in the other variable (i.e. the number of faces), note that the bound on the number of faces implies also an uniform
bound on the number of vertices. Since we need only the convexity of $G$ for polytopes with up to $N_*$
faces, let $M_*$ be the maximum number of vertices of all such polytopes.} Thus one can write the integral
\[\int_{V} |x-y|^2\d x \]
as a function of the vertices $\{v_1,\cdots,v_m\}$ only ($v_i \in \R^3$, $m\le N_*$): the cell $V$ is indeed the convex combination 
of its vertices, hence any $x\in V$ is of the form 
$x=\sum_{k=1}^m a_k v_k$. Similarly, the centroid $y := |V|^{-1}\int_V x dx$ can be also expressed in terms
of the vertices:
\begin{eqnarray*}
y & =& \frac1{|V|} \int_V x \, dx \\
&  = & \frac1{|V|}\int_{\left\{ {{\mathbf a}}\, : \, a_k\ge 0,\,  \sum_{k=1}^m a_k=1 \right\}} \, \sum_{k=1}^m a_kv_k \,\, \d { {\mathbf a}} \qquad\qquad  { {\mathbf a}:=(a_1,\cdots,a_m)}.
\end{eqnarray*}
Hence, if we define   
\begin{eqnarray*}
  I(v_1,\cdots,v_m) &:=&  \int_{V} |x-y|^2\d x   \\
& = &  \int_{\left\{{ {\mathbf a}} \, :\,  a_k\ge 0, \,  \sum_{k=1}^m a_k=1 \right\}} 
\,\,  \bigg|\,  \sum_{k=1}^m a_k v_k \, 
 - \frac1{|V|}\int_{\left\{{ \~{\mathbf a}}\, : \,  \~a_k\ge 0, \, \sum_{k=1}^m \~a_k=1 \right\}} \sum_{k=1}^m \~a_k v_k \,\, \d { \~{\mathbf a}}  \bigg|^2
 \d { {\mathbf a}}, 
\end{eqnarray*}
we see that problem reduces to convex minimization in $3 m$ variables over a convex constraint;  That is, we solve 
\[\min_{v_1,\cdots,v_m} I(v_1,\cdots,v_m)\]
under the constraint that $V$ is a convex polytope with unit volume.

\section{Conclusion and future directions}
In this paper we have shown that Voronoi cells in optimal CVTs have at most $N_*$ faces, with $N_*$ independent of the number of generators. This
allowed us to reduce Gersho's conjecture in 3D, which is intrinsically nonlocal and infinite dimensional (as it requires the number
of generators to tend to infinity), to a local and finite dimensional problem
of studying the convexity of $G$ on convex polytopes with at most $N_*$ faces. 
In our opinion, this alone is an achievement. 
However, the issue remains that the current bound on $N_*$ is far too big for computer verification. Note that 
the fact that we are interested only in the convexity of $G$  allows us to have computational errors, as long as these
are sufficiently small not to influence the convexity. 
While we have tried to optimize constants within the framework of our method, one should seek different more optimal techniques for our bounds to lower the threshold for $N_*$.

%
%
%

\bigskip
{\bf Acknowledgements:} This work was begun while 
Lu was a postdoctoral fellow at McGill University. He would like to 
thank the CRM ({\it Centre de Recherches Math\'{e}matique}) for their partial support during this period,
and Lakehead University for their partial support through its startup and RDF fundings. 
Both authors acknowledge the support of {\it NSERC} through their {\it Discovery Grants Program}.  
The authors would also like to thank David Bourne for his comments on a previous draft.

\end{document}